\documentclass[letterpaper,10 pt,conference]{ieeeconf}  

\IEEEoverridecommandlockouts                              
\usepackage{amsmath,amsthm,amssymb}
\usepackage{amsfonts}
\usepackage{dsfont}
\usepackage{graphicx}
\usepackage{color}
\usepackage{cite}
\usepackage{subcaption}
\usepackage{flushend}
\usepackage{graphicx}
\usepackage{algorithm,algorithmic}
\usepackage{epstopdf}
\usepackage{bm}
\usepackage{stmaryrd}
\newtheorem{theorem}{Theorem}

\newtheorem{lemma}{Lemma}
\newtheorem{corollary}{Corollary}
\newtheorem{definition}{Definition}

\DeclareMathOperator*{\argmin}{arg\,min}

\title{\LARGE \bf
A Convex Approach to Frisch-Kalman Problem
}

\author{Di Zhao, Anders Rantzer, and Li Qiu
\thanks{*This work was supported in part by the Research Grants Council of Hong Kong Special Administrative Region, China, under Theme-based Research Scheme T23-701/14N, and by the Swedish Research Council, Sweden.}
\thanks{D. Zhao and L. Qiu are with the Department of Electronic \& Computer Engineering, The Hong Kong University of Science and Technology, Clear Water Bay, Kowloon, Hong Kong, China. Emails: {\tt\small dzhaoaa@ust.hk}, {\tt\small eeqiu@ust.hk}}
\thanks{A. Rantzer is with the Department of Automatic Control and the Excellence Center ELLIIT, Lund University, Lund, Sweden. Email:~{\tt\small rantzer@control.lth.se}
}
}
\graphicspath{{./fig/}}

\begin{document}

\maketitle
\thispagestyle{empty}
\pagestyle{empty}

\begin{abstract}
This paper proposes a convex approach to the Frisch-Kalman problem that identifies the linear relations among variables from noisy observations.  The problem was proposed by Ragnar Frisch in 1930s, and was promoted and further developed by Rudolf Kalman later in 1980s. It is essentially a rank minimization problem with convex constraints. Regarding this problem, analytical results and heuristic methods have been pursued over a half century. The proposed convex method in this paper is shown to be accurate and demonstrated to outperform several commonly adopted heuristics when the noise components are \textcolor{black}{relatively small compared with the underlying data}.
\end{abstract}

\section{Introduction}
The identification from noisy data has become an important problem of statistics and, via applications, of econometrics, biometrics, psychometrics and so on. Among various problems with different models on the data and noise, the Frisch-Kalman problem (scheme) \cite{frisch1934statistical,kalman1982system,Kalman1985RankProblem}, \textcolor{black}{which is rooted in the work of Charles Spearman \cite{spearman1904} in 1904,} has attracted much attention and been investigated since 1930s \cite{frisch1934statistical,ledermann1940,kalman1982system,Kalman1985RankProblem,shapiro1982rank,anderson1987dynamic,soderstrom2007errors,ning2015linear}.

Given a finite family of $n$ (random) variables $\{\omega_1,\omega_2,\dots,\omega_n\}$ that are linearly dependent, we call them the true or underlying data, and in general, we have no direct access to their exact values. Instead, we can measure or observe their values in a noisy environment. The observed data $\{{x}_1,{x}_2,\dots,{x}_n\}$ are corrupted by noise variables $\{\delta_1,\delta_2,\dots,\delta_n\}$ additively, i.e.,
$$x_i=\omega_i+\delta_i,~i=1,2,\dots,n.$$
 One may ask naturally: can we identify the linear relations among the true data from the observed (noisy) data samples? For this purpose, what else do we need to know about the data and noise? A well-established answer to the problems is given by the Frisch-Kalman scheme.

 Denote by $\Sigma$ the covariance matrix of the observed data $\{x_i\}$, which may be obtained from repeated experiments and measurements. Denote by $\Omega$ and $\Delta$ the covariance matrices of the true data $\{\omega_i\}$ and noise $\{\delta_i\}$, respectively. The key assumption in the Frisch-Kalman scheme is that the noise components are mutually uncorrelated and independent from the true data, in which case the following decomposition holds:
 $$\Sigma=\Omega+\Delta,$$
 and $\Delta$ is nonnegative and diagonal. \textcolor{black}{Such a decomposition is called the factor analytic decomposition \cite{comrey1992first,shapiro1982rank} as is used in a statistical method --- the factor analysis.} The Frisch-Kalman scheme suggests one way to identify the linear relations via the minimization of the rank of $\Sigma-\Delta$ over all possible \textcolor{black}{noise covariance matrices $\Delta$}. \textcolor{black}{Regarding this scheme, a particularly important problem, which aims at finding the exact class of the observed covariance matrices $\Sigma$ such that the maximum corank of $\Sigma-\Delta$ over all $\Delta$ is one, has been investigated since 1940s \cite{reiersol1941,Kalman1985RankProblem,kalman1982system,shapiro1982rank,anderson1987dynamic,ning2015linear}.}

The Frisch-Kalman problem is essentially a rank minimization problem with convex constraints. It is closely related to the low-rank matrix completion problem \cite{Candes2010MatrixCompletion,Recht2010,jain2013low}, where one wishes to complete a partially known matrix so that its rank is as small as possible. The nuclear norm minimization \cite{Recht2010} has been pursued as a suitable heuristic for general rank minimization problems. In terms of the Frisch-Kalman problem, the nuclear norm heuristic reduces to the well-studied minimum trace factor analysis \cite{ledermann1940,shapiro1982rank,ning2015linear,saunderson2012diagonal}. As generalizations to the nuclear norm, a family of low-rank inducing norms, \textcolor{black}{called the $r*$-norms \cite{grussler2018low,grussler_pontus2018low} or spectral $r$-support norms \cite{mcdonald2016new},} have been recently proposed, which improve the performance of the nuclear norm heuristic for rank minimization problems. In addition to the low-rank inducing norms, other surrogates have been studied for the rank function, for example, the logarithm of the determinant (log-det) \cite{fazel2003log}.

In this paper, we propose a convex approach to the Frisch-Kalman problem by first reformulating the problem into a norm minimization problem with a rank constraint, then relaxing it into a convex problem that is essentially a semi-definite programming (SDP) \cite{boyd2004convex}. Conditions on the tight relaxation are developed and demonstrated.
The reformulated Frisch-Kalman problem additionally penalizes the variances of noise components, which is motivated by the application scenarios when the noise are well-bounded with respect to the underlying data. For example, population census and mapping in developed countries \cite{Andrew2011PopulationCensus}, channel estimations in slow fading channels \cite{tse2005fundamentals}, long-term global surface temperature measure \cite{Hansen14288}, and so on. Comparisons with the existing heuristic methods, including the nuclear norm minimization \cite{Recht2010}, the $r*$-norm minimization \cite{grussler2018low} and the log-det heuristic \cite{fazel2003log}, show that the proposed method has high success rates and strictly outperforms the others when the noise components are well bounded with respect to the underlying data.

The rest of the paper is organized as follows. In Section~\ref{sec:pre}, basic notation and preliminary results are introduced. In Section~\ref{sec:proposed_method}, the main algorithm is developed. In Section~\ref{sec:tight_relax}, the conditions on the tight relaxation are obtained. In Section~\ref{sec:compare_Heuristic}, comparisons with the existing heuristic methods are shown via simulations. Finally, in Section~\ref{sec:conclusion}, the study is concluded and future research directions are introduced.

\section{Preliminaries}\label{sec:pre}
\subsection{Notation}
Let $\mathbb{R}$ be the real field,  and $\mathbb{R}^n$ be the linear space of $n$-dimensional vectors over $\mathbb{R}$. For $x \in \mathbb{R}^n$, its Euclidean norm is denoted by $\|x\|$.

For matrix $A\in \mathbb{R}^{m \times n}$, its element in the $i$th row and $j$th column is denoted by $[A]_{ij}$, $i=1,2,\dots,m$, $j=1,2,\dots,n$, its transpose is by $A^T$, its range is by \textcolor{black}{$$\mathcal{R}(A):=\{y\in\mathbb{R}^m|~ y=A x~\text{for some}~x\in\mathbb{R}^n\},$$
its kernel is by}
$$\textcolor{black}{\mathcal{K}(A):=\{x\in\mathbb{R}^n |~ A x=0 \},}$$
and its $k$th singular value is by $\sigma_k(A)$, $k=1,2,\dots  ,l$, in a nonincreasing order, where $l=\min\{m,n\}$. The largest and smallest singular values are specially denoted by $\bar{\sigma}(A):=\sigma_1(A)$ and $\underline{\sigma}(A):=\sigma_{l}(A)$, respectively. The operator norm (spectral norm) and the Frobenius norm of $A$ are respectively denoted by $$\|A\|:=\bar{\sigma}(A)~\text{and}~\|A\|_F:=\sqrt{\sum_{k=1}^{l}\sigma_k^2(A)}.$$
The $r$-norm \cite{grussler2018low} of $A$, $r=1,2,\dots,l$, is defined via
$$\|A\|_r:=\sqrt{\sum_{i=1}^r \sigma^2_{i}(A)}.$$
Clearly, $\|A\|_F=\|A\|_l.$ Denote its singular value decomposition (SVD) as
$$A=USV^T=\sum_{i=1}^{l}\sigma_i(A)u_iv_i^T,$$
where $U,V$ are unitary. For $A,B\in \mathbb{R}^{m \times n}$, their inner product is defined via
$$\langle A,B\rangle:=\text{tr}(A^TB).$$
For $X\in\mathbb{R}^{n\times n}$, the diagonal matrix that keeps the diagonal terms of $X$ is denoted by $\text{diag}(X)$. For $x\in\mathbb{R}^n$, the diagonal matrix with its $i$th diagonal term given by $x_i$ is denoted by $\text{diag}^*(x)$.

Some frequently used special sets of matrices are as follows.
\begin{itemize}
	\item Denote by $\mathbb{S}^n$ the set of all symmetric matrices in $\mathbb{R}^{n\times n}$.
	\item Denote by $\mathbb{S}^n_0$ the set of symmetric matrices in $\mathbb{R}^{n\times n}$ with zero diagonals.
	\item Denote by $\mathbb{S}^n_+$ ($\mathbb{S}^n_{++}$, resp.) the set of all positive semi-definite (definite, resp.) matrices in $\mathbb{R}^{n\times n}$.
	\item Denote by $\mathbb{D}^n$ the set of all diagonal matrices in $\mathbb{R}^{n\times n}$.
    \item Denote by $\mathbb{D}_+^n$ the set of all nonnegative diagonal matrices in $\mathbb{R}^{n\times n}$.
\end{itemize}
\subsection{Standard Low-Rank Approximation}\label{subsec:lowrank_approx}
Let $A\in \mathbb{R}^{m \times n}$ and $l=\min\{m,n\}$. Consider the following standard rank approximation problem:
\begin{align}\label{eq:standard_problem}
\min_{{B}}\left\{\left\|A-{B}\right\|_F|~\text{rank}({B})\leq r\right\},
\end{align}
for $r=1,2,\dots,l$. Based on the Schmidt-Mirsky theorem \cite[Chapter~IV]{stewart1990matrix}, all solutions to problem~(\ref{eq:standard_problem}) is given by
$$\text{svd}_r(A)\hspace*{-2pt}:=\hspace*{-3pt}\left\{\sum_{i=1}^r\sigma_i(A)u_iv_i^T\Bigg|~A=\sum_{i=1}^{l}\sigma_i(A)u_iv_i^T~\text{is SVD}\hspace*{-2pt}\right\},$$
which is called the set of all standard $r$th order SVD-approximation to $A$. Clearly, for every ${B}\in\text{svd}_r(A)$, it holds that
\begin{align}\label{eq:rank_Ahat}
\text{rank}({B})=r.
\end{align}
When $\sigma_r(A)>\sigma_{r+1}(A)$, $\text{svd}_r(A)$ is a singleton and its only element is denoted by
$$[A]_r:=\sum_{i=1}^r\sigma_i(A)u_iv_i^T.$$
The optimal value to the problem is given by
\begin{multline*}
\min_{{B}}\left\{\left\|A-{B}\right\|_F|\ \text{rank}({B})\leq r\right\}=\\
\left\|\begin{bmatrix}\sigma_{r+1}(A)&\cdots&\sigma_{l}(A)\end{bmatrix}\right\|=\sqrt{\|A\|^2_F-\|A\|_r^2}.
\end{multline*}
\subsection{Frisch-Kalman Problem}
The Frisch-Kalman problem is defined via the following optimization \cite{frisch1934statistical,kalman1982system,Kalman1985RankProblem,ning2015linear}.
\begin{definition}\label{def:FrischKalman_problem}
	Given $\Sigma\in\mathbb{S}^n_{++}$, determine
	\begin{equation}\label{eq:FrischKalman}
    \begin{aligned}
	\text{\rm mr} (\Sigma)=\min_{\Omega,\Delta}\{\text{\rm rank}(\Omega)|~\Sigma = \Omega+\Delta,\\
	\Omega\in\mathbb{S}^n_{+}, {\Delta}\in\mathbb{D}_+^n\}.
	\end{aligned}
	\end{equation}
\end{definition}
A matrix $\Delta$ is said to be feasible to problem~(\ref{eq:FrischKalman}), if $\Delta$ is diagonal and $\Sigma\geq \Delta \geq 0$. A trivial upper bound to the problem is given by $\text{mr}(\Sigma)\leq n-1$, which can be obtained by selecting a feasible $\Delta=\underline{\sigma}(\Sigma)I$. The Frisch-Kalman problem is, in general, non-convex, and many heuristic convex approaches have been proposed and investigated \cite{Recht2010,jain2013low,shapiro1982rank,ning2015linear}.

\section{Proposed Convex Approach}\label{sec:proposed_method}
In this section, we develop a convex approach to solving the Frisch-Kalman problem, and further apply the algorithm to a variant of the Frisch-Kalman problem, called the Shapiro problem \cite{shapiro1982rank}.
\subsection{Reformulation and Relaxation}
Consider the factor analytic decomposition $\Sigma=\Omega+\Delta$. In the context of Frisch-Kalman scheme, $\Omega\in\mathbb{S}_+^n$ is the unknown covariance matrix of some linearly dependent true data variables, and hence it is expected to have a low rank. The matrix $\Delta$ is the covariance matrix of an uncorrelated noise vector, and hence it must be nonnegative diagonal. Finally, $\Sigma\in\mathbb{S}_{++}^n$, as the sum of $\Omega$ and $\Delta$, is the covariance matrix of the noisy data under the assumption that the data and noise are independent. In many practical situations, the variances of the noise may be much smaller than those of the true data. To take advantage of the additional preknowledge, we may penalize the ``size of noise'' as in the following reformulation of Frisch-Kalman problem.

Given an integer $r\in[1,n]$, we reformulate the Frisch-Kalman problem into the following norm minimization problem with a rank constraint:
\begin{equation}\label{eq:reformulated_FK}
\begin{aligned}
\min_{\Omega} \{\|\Sigma-\Omega\|_F^2|~\text{rank}(\Omega)\leq r,~ \Sigma\geq \Omega\geq 0,\\~\Sigma-\Omega\in\mathbb{D}^n\}.
\end{aligned}
\end{equation}
Here, the object function is simply the sum of squares of all the entries in the diagonal matrix $\Sigma-\Omega$. The rank function is moved from the object function in (\ref{eq:FrischKalman}) to the constraints in (\ref{eq:reformulated_FK}). If this problem is feasible, then we obtain immediately $\text{mr}(\Sigma)\leq r$. In other words, we can search for $\text{mr}(\Sigma)$ via solving a sequence of feasibility problems of (\ref{eq:reformulated_FK}) with different levels of $r\in[1,n]$. However, the reformulated problem (\ref{eq:reformulated_FK}) is still non-convex. To proceed, we develop some further relaxations in the following.

We introduce a symmetric matrix with zero diagonals, namely, $\Lambda\in \mathbb{S}_0^n$, as the dual variable. Based on the reformulated Frisch-Kalman problem (\ref{eq:reformulated_FK}), we have the following series of equalities and inequalities:
\begin{align}\label{eq:relaxed_FK}
&\min_{\Omega} \{\|\Sigma-\Omega\|_F^2|~\text{rank}(\Omega)\leq r,\Sigma\geq \Omega\geq 0,\Sigma-\Omega\in\mathbb{D}^n\} \nonumber\\
&= \min_{\Omega}\max_\Lambda \{\|\Sigma-\Omega\|_F^2 + 2 \langle \Lambda,\Sigma-\Omega\rangle|\nonumber\\
&\hspace{90pt}\text{rank}(\Omega)\leq r,~\Sigma\geq \Omega\geq 0, \Lambda\in\mathbb{S}_0^n\} \nonumber\\
&\geq \max_\Lambda\min_{\Omega} \{\|\Sigma-\Omega\|_F^2 + 2 \langle \Lambda,\Sigma-\Omega\rangle|\nonumber\\
&\hspace{90pt}\text{rank}(\Omega)\leq r,~\Sigma\geq \Omega\geq 0, \Lambda\in\mathbb{S}_0^n\} \nonumber\\
&\geq\max_\Lambda\min_{\Omega} \{\|\Sigma+\Lambda-{\Omega}\|_F^2-\|\Sigma+\Lambda\|_F^2+2\langle \Lambda,\Sigma\rangle\nonumber\\
 &\hspace{102pt}+ \|\Sigma\|_F^2|~\text{rank}(\Omega)\leq r,\Lambda\in\mathbb{S}_0^n\}\nonumber\\
&=\max_\Lambda \{-\|\Sigma+\Lambda\|_r^2+2\langle \Lambda,\Sigma\rangle + \|\Sigma\|_F^2|~\Lambda\in\mathbb{S}_0^n\},
\end{align}
where the first equality is due to that the maximization over $\Lambda\in\mathbb{S}_0^n$ forces $\Sigma-\Omega$ to be diagonal, the first inequality follows from the max-min inequality, the last inequality is due to that the constraint of $\Sigma\geq \Omega\geq 0$ is removed, and the last equality follows from the standard SVD-approximation to $\Sigma+\Lambda$ in Section~\ref{subsec:lowrank_approx}.

Here, the problem (\ref{eq:relaxed_FK}) is a maximization of a concave function with convex constraints, hence it is a convex problem, which is our targeted convex relaxation to the original non-convex problem. Using similar tricks in \cite{grussler2018low}, we can equivalently transform (\ref{eq:relaxed_FK}) into the following SDP:
\begin{align}\label{eq:SDP_FK}
\max_{T,\Lambda,\gamma} &-\text{tr}(T)-\gamma(n-r)+2\langle \Lambda,\Sigma\rangle+\|\Sigma\|_F^2\nonumber\\
\text{s.t.}~&\Lambda\in\mathbb{S}^n_0,~T- \gamma I\in\mathbb{S}_+^n,\\
~&\begin{bmatrix}
    T & \Sigma+\Lambda \\
    \Sigma+\Lambda & I \\
  \end{bmatrix}\in\mathbb{S}_+^{2n}.\nonumber
\end{align}

\subsection{Proposed Algorithm}
Suppose we have solved the SDP in (\ref{eq:SDP_FK}) and obtained an optimal dual variable $\Lambda^\star$. What is the most appropriate value for the primal variable $\Omega$ based on the dual optimum? The following theorem shows how we obtain the optimal primal variable $\Omega^\star$ when the duality gap is zero, i.e.,
\begin{multline}\label{eq:primal_eq_dual}
	\min_{\Omega} \{\|\Sigma-\Omega\|_F^2|~ \text{\rm rank}(\Omega)\leq r,~\Sigma\geq \Omega\geq 0,~\Sigma-\Omega\in\mathbb{D}^n\} \\
	=\max_\Lambda \{-\|\Sigma+\Lambda\|_r^2+2\langle \Lambda,\Sigma\rangle + \|\Sigma\|_F^2|~\Lambda\in\mathbb{S}_0^n\}.
\end{multline}
\begin{theorem}\label{thm:proposed_recover}
	Let $\Sigma\in\mathbb{S}^n_{++}$ and equality (\ref{eq:primal_eq_dual}) be true. Then a solution to (\ref{eq:reformulated_FK}) satisfies
\begin{align}\label{eq:Hstar}
\Omega^\star\in\text{\rm svd}_r(\Sigma+\Lambda^\star),
\end{align}
	where $\Lambda^\star$ solves (\ref{eq:SDP_FK}).
\end{theorem}
\begin{proof}
	Since equality~(\ref{eq:primal_eq_dual}) is true, all the inequalities above (\ref{eq:relaxed_FK}) are actually equalities. Hence a solution to (\ref{eq:reformulated_FK}) necessarily solves the following problem:
\begin{multline}
\min_{\Omega} \{\|\Sigma+\Lambda^\star-{\Omega}\|_F^2-\|\Sigma+\Lambda^\star\|_F^2+2\langle \Lambda^\star,\Sigma\rangle\\
 + \|\Sigma\|_F^2|~\text{rank}(\Omega)\leq r\}.
\end{multline}
By the standard SVD-approximation shown in (\ref{eq:standard_problem}), the solution to (\ref{eq:reformulated_FK}) satisfies
	$$\Omega^\star\in\text{\rm svd}_r(\Sigma+\Lambda^\star),$$
which completes the proof.
\end{proof}
As we see from the theorem, $\Omega^\star$ may be selected \textcolor{black}{as an appropriate candidate to test the feasibility of (\ref{eq:reformulated_FK}). In this case, we may first obtain an $\Omega^\star\in\text{\rm svd}_r(\Sigma+\Lambda^\star)$,} then check whether $\Omega^\star$ is feasible to (\ref{eq:reformulated_FK}). It is clear that $\text{rank}(\Omega^\star)=r$ due to (\ref{eq:rank_Ahat}), hence it suffices to check whether $\Sigma\geq \Omega^\star\geq 0$ and $\Sigma-\Omega^\star\in\mathbb{D}^n$.

Based on the above developments, we propose the following algorithm involving only convex optimizations to solve the Frisch-Kalman problem.
\begin{algorithm}[H]\label{alg:proposed}
	\caption{Proposed Method to Frisch-Kalman Problem}
	\begin{description}
		\item [\textbf{Step~1}] ~Given $\Sigma\in\mathbb{S}^n_{++}$. Set the initial searching rank as $r\in[1,n-1]$.
		\item [\textbf{Step~2}] ~Compute $\Lambda^\star$ via the SDP in (\ref{eq:SDP_FK}).
		\item [\textbf{Step~3}] ~Compute an $\Omega^\star\in\text{svd}_r(\Sigma+\Lambda^\star)$. Check whether $\Sigma\geq \Omega^\star\geq 0$ and $\Sigma-\Omega^\star\in\mathbb{D}^n$. If not, let $r:=r+1$ and go to \textbf{Step~2}.
		\item [\textbf{Step~4}] ~\textcolor{black}{An upper bound of the Frisch-Kalman problem (\ref{eq:FrischKalman}) is obtained as $r^\star:=\text{rank}(\Omega^\star)\geq\text{mr}(\Sigma)$.}
	\end{description}
\end{algorithm}

\subsection{Application to Shapiro Problem}
Consider the following variant of the Frisch-Kalman problem, called the Shapiro problem \cite{shapiro1982rank}, where the constraint that $\Delta$ is nonnegative is relaxed. Investigation into such a relaxed problem brings about more direct understanding on how the off-diagonal entries of $\Sigma$ affect the minimization of its rank.
\begin{definition}[Shapiro Problem]
	Given $\Sigma\in\mathbb{S}^n_{++}$, determine
	\begin{equation}\label{eq:Shapiro_problem}
	\begin{aligned}
	\text{\rm mr}_s(\Sigma) =\min_{\Omega,\Delta}\{\text{\rm rank}(\Omega)| ~\Sigma = \Omega+\Delta,\\
	\Omega\in\mathbb{S}_+^n, {\Delta}\in\mathbb{D}^n
	\}.
	\end{aligned}
	\end{equation}
\end{definition}
Actually, Shapiro and Frisch-Kalman problems share many similar properties. Naturally, we can apply the above algorithm to Shapiro problem with slight modifications, i.e., replacing \textbf{Step~3} with\\
\textbf{Step~3$^*$} \hspace{1pt}Compute $\Omega^\star\in\text{svd}_r(\Sigma+\Lambda^\star)$. Check whether $\Omega^\star\in\mathbb{S}^n_+$ and $\Sigma-\Omega^\star\in\mathbb{D}^n$. If not, let $r:=r+1$ and go to \textbf{Step~2}.

The obtained rank $r^\star$ satisfies that $\text{mr}_s(\Sigma)\leq r^\star$.
\subsection{Extension to the Complex-Valued Case}\label{sec:complex_case}
Denote by $\mathbb{H}_{+}$ ($\mathbb{H}_{++}$, resp.) the set of all positive semi-definite (definite, resp.) matrices in $\mathbb{C}^{n\times n}$.
The Frisch-Kalman problem can be extended to the case with complex-valued matrices as follows.
\begin{definition}[Complex-Valued Frisch-Kalman Problem]
	Given $\Sigma\in\mathbb{H}^n_{++}$, determine
	\begin{equation}\label{eq:Shapiro_problem}
	\begin{aligned}
	\text{\rm mr}(\Sigma) =\min_{\Omega,\Delta}\{\text{\rm rank}(\Omega)| ~\Sigma = \Omega+\Delta,\\
	\Omega\in\mathbb{H}_+^n, {\Delta}\in\mathbb{D}^n_+
	\}.
	\end{aligned}
	\end{equation}
\end{definition}
In this case, we may directly apply Algorithm~1 to the above problem by suitably replacing all the involved symmetric matrices with the Hermitian ones.

\section{Conditions on Tight Relaxation}\label{sec:tight_relax}

From the previous developments, we know  (\ref{eq:relaxed_FK}) is a convex optimization problem  and its optimal value is a lower bound of that of the reformulated Frisch-Kalman problem (\ref{eq:reformulated_FK}). One may ask naturally how tight the bound is, or how tight the convex relaxation is. In general, (\ref{eq:relaxed_FK}) is not equivalent to (\ref{eq:reformulated_FK}), but we will show that within a certain class of $\Sigma$, the solutions to (\ref{eq:relaxed_FK}) will also solve (\ref{eq:reformulated_FK}) via $\Omega^\star\in\text{svd}_r(\Sigma+\Lambda^\star)$, hence the duality gap is zero.

We start with the following lemma, where $\hat{\Omega}$ may be viewed as the covariance of the underlying data. We remove the requirement that $\Sigma>0$ temporarily.
\begin{lemma}\label{lem:given_low_rank}
	Let $\Sigma=\hat{\Omega}\in\mathbb{S}^n_+$ and $\text{\rm rank}(\hat{\Omega})=r$. Then the solutions to (\ref{eq:relaxed_FK}) solve (\ref{eq:reformulated_FK}), and equality~(\ref{eq:primal_eq_dual}) holds
	with optima attained on $\Lambda^\star=0$ and $\Omega^\star=\hat{\Omega}$.
\end{lemma}
\begin{proof}
	Since $\hat{\Omega}\in\mathbb{S}^n_+$ and $\text{\rm rank}(\hat{\Omega})=r$, it can be unitarily diagonalized, i.e.,
	$\hat{\Omega}=USU^T$ where
    $$S=\begin{bmatrix}
          S_r & 0 \\
          0  & 0_{n-r} \\
        \end{bmatrix}\in\mathbb{D}_+^n,~S_r >0~\text{and}~U~\text{is unitary}.$$
    Let $\Lambda\in\mathbb{S}_0^n$ and partition $U^T\Lambda U$ conformably with $S$ via
	$$U^T\Lambda U=\begin{bmatrix}
	L_{11} & L_{12}\\
	L_{21} & L_{22}
	\end{bmatrix}.$$
	Hence,
	\begin{align*}
    -\|&\hat{\Omega}+\Lambda\|_r^2+2\langle \Lambda,\hat{\Omega}\rangle + \|\hat{\Omega}\|_F^2\\
	&=-\|S+U^T\Lambda U\|_r^2+ 2\langle U^T\Lambda U,S\rangle + \|\hat{\Omega}\|_r^2\\
	&=-\left\|\begin{bmatrix}
	L_{11}+S_r & L_{12}\\
	L_{21} & L_{22}
	\end{bmatrix}\right\|_r^2 + 2\langle L_{11},S_r\rangle + \|\hat{\Omega}\|_F^2\\
	&\leq -\|L_{11}+S_r\|_F^2+ 2\langle L_{11},S_r\rangle + \|\hat{\Omega}\|_F^2\\
	&=-\|S_r\|_F^2-\|L_{11}\|_F^2+\|\hat{\Omega}\|_F^2\\
	&\leq -\|S_r\|_F^2+\|\hat{\Omega}\|_F^2=0,
	\end{align*}
	where the first inequality follows from \cite[Theorem~4.4]{stewart1990matrix}
	and the inequalities become equalities when $\Lambda=0$. It is clear that
	\begin{multline*}
    0\leq \min_{\Omega} \left\{\|\hat{\Omega}-{\Omega}\|_F^2\Big|~ \text{\rm rank}(\Omega)\leq r,\right.\\
    \left.\hat{\Omega}\geq \Omega\geq 0,~\hat{\Omega}-\Omega\in\mathbb{D}^n\right\}\leq \|\hat{\Omega}-\hat{\Omega}\|_F^2=0.
    \end{multline*}
	Therefore, equality (\ref{eq:primal_eq_dual}) is true with optima attained on $\Lambda^\star=0$ and $\Omega^\star=\hat{\Omega}$.
\end{proof}
The lemma gives us an intuition that when the low-rank matrix $\hat{\Omega}$ is slightly perturbed by a diagonal matrix $\Delta$, the optimum of (\ref{eq:relaxed_FK}) is very likely to be attained on $\Lambda^\star$ that is close to zero. The underlying reason is that $\|\Delta\|_F$ is close to zero, and hence we expect certain ``continuity'' properties, considering that a rank function is not continuous at all. Furthermore, the obtained $\Lambda^\star$ might solve (\ref{eq:reformulated_FK}) via $\Omega^\star\in\text{svd}_r(\Sigma+\Lambda^\star)$. The intuition is not completely true for the most general case, but the underlying idea helps develop the following conditions on the tight relaxation.

Given $\hat{\Omega}\in\mathbb{S}^n_{+}$ with $\text{rank}(\hat{\Omega})=r$, define a set of diagonal matrices via
\begin{multline}\label{eq:perturbed_diag}
\mathcal{D}_{\hat{\Omega}}:=\Big\{{\Delta}\in\mathbb{D}_+^n\Big|~\exists~\Lambda\in\mathbb{S}_0^n,~\text{such that}\Big. \\ \Big.~\|\Delta+\Lambda\|<\sigma_r(\hat{\Omega}),~\mathcal{R}(\Delta+\Lambda)\perp\mathcal{R}(\hat{\Omega})\Big\}.
\end{multline}
In correspondence, define the following set of positive definite matrices
\begin{align}\label{eq:perturbed_Sigma}
\mathcal{S}_{\hat{\Omega}}:=\left\{\Sigma=\hat{\Omega}+\Delta>0\Big|~\Delta\in\mathcal{D}_{\hat{\Omega}}\right\}.
\end{align}
Obviously, $0\in\mathcal{D}_{\hat{\Omega}}$. In addition it is not hard to verify that if ${\Delta}\in\mathcal{D}_{\hat{\Omega}}$, we have $\alpha \Delta\in\mathcal{D}_{\hat{\Omega}}$ for all $\alpha\in[0,1]$. Together with condition $\|\Delta+\Lambda\|<\sigma_r(\hat{\Omega})$, we know that $\mathcal{D}_{\hat{\Omega}}$ characterizes a neighborhood of diagonal matrices with ``small'' norms. In the definition, the condition of $ \mathcal{R}(\Delta+\Lambda)\perp\mathcal{R}(\hat{\Omega})$ has the following implication.
\begin{lemma}\label{lem:implication}
	Let $\hat{\Omega}\in\mathbb{S}^n_{+}$ with $\text{\rm rank}(\hat{\Omega})=r$. If $X\in\mathbb{S}^n$, $\|X\|<\sigma_r(\hat{\Omega})$ and $\mathcal{R}(X)\perp\mathcal{R}(\hat{\Omega}) $, then
	$$[\hat{\Omega}+X]_r=\hat{\Omega}.$$
\end{lemma}
\begin{proof}
	Applying SVD on $\hat{\Omega}$, we obtain
	$$\hat{\Omega}=\sum_{i=1}^r \sigma_i(\hat{\Omega})u_iu_i^T.$$
	Since $\mathcal{R}(X)\perp\mathcal{R}(\hat{\Omega})$, it follows that $\text{rank}(X)\leq n-r$. Applying SVD on $X$, we obtain
	$$X=\sum_{i=1}^{n-r}\sigma_i(X)w_iv_i^T.$$
	Since $X$ is symmetric, $w_i=\pm v_i$. Again from $\mathcal{R}(X)\perp\mathcal{R}(\hat{\Omega})$, we know $\langle u_i,v_j \rangle=0$ for all $i=1,2,\dots,r$ and $j=1,2,\dots,n-r$. As a result, it follows from $\bar{\sigma}(X)<\sigma_r(\hat{\Omega})$ that
	$$\hat{\Omega}+X=\sum_{i=1}^r \sigma_i(\hat{\Omega})u_iu_i^T+\sum_{j=1}^{n-r}\sigma_j(X)w_jv_j^T$$
	is an SVD for $\hat{\Omega}+X$, and
	$$[\hat{\Omega}+X]_r=\sum_{i=1}^r \sigma_i(\hat{\Omega})u_iu_i^T=\hat{\Omega}.$$
	This completes the proof.
\end{proof}
\textcolor{black}{In general, $\mathcal{S}_{\hat{\Omega}}$ is non-empty, as shown in the following lemma.
\begin{lemma}\label{lem:nonempty_S}
	Let $\hat{\Omega}\in\mathbb{S}^n_{+}$ with $\text{\rm rank}(\hat{\Omega})=r<n$. Then
	$$\mathcal{S}_{\hat{\Omega}}\neq\emptyset.$$
\end{lemma}
\begin{proof}
Without loss of generality, we may assume that
$$\hat{\Omega}=\begin{bmatrix}  \Omega_0 & 0\\ 0 & \Omega_1\end{bmatrix},$$
where $\Omega_0\in\mathbb{D}^k\cap\mathbb{S}^k_{++}$ with $0\leq k\leq r$, and  $\Omega_1\in\mathbb{S}_+^{n-k}$ is a block diagonal matrix with either zero blocks or irreducible blocks whose sizes are no less than two.
In this case, there exists
$$v=\begin{bmatrix}0\\ v_1\end{bmatrix}\in\mathcal{K}(\hat{\Omega})\setminus\{0\}$$
where $v_1\in\mathbb{R}^{n-k}$ is element-wise nonzero and $\|v\|<\sqrt{\sigma_r(\hat{\Omega})}$. Construct that $\Delta=\text{diag}(vv^T)\in\mathbb{D}_+$ and $\Lambda=vv^T-\Delta$. It is straightforward to verify that $\Delta\in\mathcal{D}_{\hat{\Omega}}$ and $[\Delta]_{ii}=[vv^T]_{ii}> 0$ for all $i=k+1,\dots,n$. It follows that
$$\Sigma:=\hat{\Omega}+\Delta=\begin{bmatrix}  \Omega_0 & 0\\ 0 & \Omega_1+\text{diag}(v_1v_1^T)\end{bmatrix}\in\mathbb{S}^n_{++},$$
whereby $\Sigma\in\mathcal{S}_{\hat{\Omega}}$.
\end{proof}
Given every low-rank matrix $\hat{\Omega}$, we obtain a neighborhood of noisy covariance matrices $\Sigma\in\mathcal{S}_{\hat{\Omega}}$ centered at $\hat{\Omega}$ by the above lemma. The following result shows that for each $\Sigma$ in $\mathcal{S}_{\hat{\Omega}}$,  the duality gap between optimization problems in (\ref{eq:reformulated_FK}) and (\ref{eq:relaxed_FK}) is zero.}
\begin{theorem}\label{thm:tight_relax}
	Let $\hat{\Omega}\in\mathbb{S}^n_{+}$ with $\text{\rm rank}(\hat{\Omega})=r<n$ and $\Sigma\in\mathcal{S}_{\hat{\Omega}}$.
	Then the solutions to (\ref{eq:relaxed_FK}) solve (\ref{eq:reformulated_FK}), and equality~(\ref{eq:primal_eq_dual}) holds
	with optima attained on $\Omega^\star=\hat{\Omega}$ and $\Lambda^\star$ satisfying $$\|\Lambda^\star+\Sigma-\hat{\Omega}\|<\sigma_r(\hat{\Omega})~\text{and}~\mathcal{R}(\Lambda^\star+\Sigma-\hat{\Omega})\perp\mathcal{R}(\hat{\Omega}).$$
\end{theorem}
\begin{proof}
	Decompose matrix $\Sigma$ into $\Sigma=\hat{\Omega}+\Delta$ with $\Delta\in\mathcal{D}_{\hat{\Omega}}$. Then there exists a $\Lambda^\star\in\mathbb{S}^n_0$ satisfying that $\mathcal{R}( \Delta+\Lambda^\star)\perp \mathcal{R}(\hat{\Omega})$ and that $\|\Delta+\Lambda^\star\|<\sigma_r(\hat{\Omega})$. Furthermore, it follows from Lemma~\ref{lem:implication} that
	$$[\Sigma+\Lambda^\star]_r=[\hat{\Omega}+\Lambda^\star+\Delta]_r=\hat{\Omega}.$$
	Again from $\mathcal{R}( \Delta+\Lambda^\star)\perp \mathcal{R}(\hat{\Omega})$, we obtain that $$\langle \Sigma-\hat{\Omega}+\Lambda^\star,\hat{\Omega}\rangle=\langle \Delta+\Lambda^\star,\hat{\Omega}\rangle =0.$$
	Using the above equalities, we have
	\begin{multline*}
	-\|\Sigma+\Lambda^\star\|_r^2+2\langle \Lambda^\star,\Sigma\rangle + \|\Sigma\|_F^2\\
    =-\|\hat{\Omega}\|_F^2-2\langle \Delta,\hat{\Omega}\rangle +\|\Sigma\|_F^2=\|\Delta\|_F^2.
	\end{multline*}
	Moreover, it follows from the inequalities above (\ref{eq:relaxed_FK}) that
	\begin{align*}
	&\|\Delta\|_F^2=\|\Sigma-\hat{\Omega}\|_F^2\\
	&\geq\min_{\Omega} \{\|\Sigma-\Omega\|_F^2|~\\
    &\hspace{40pt}\text{\rm rank}(\Omega)\leq r,\Sigma\geq \Omega\geq 0,\Sigma-\Omega\in\mathbb{D}^n\}\\
	&\geq \max_\Lambda \{-\|\Sigma+\Lambda\|_r^2+2\langle \Lambda,\Sigma\rangle + \|\Sigma\|_F^2|~\Lambda\in\mathbb{S}_0^n\} \\
	&\geq -\|\Sigma+\Lambda^\star\|_r^2+2\langle \Lambda^\star,\Sigma\rangle + \|\Sigma\|_F^2 =\|\Delta\|_F^2.
	\end{align*}
	As a result, the optima of (\ref{eq:reformulated_FK}) and (\ref{eq:relaxed_FK}) are attained on $\Omega^\star=\hat{\Omega}$ and a desired $\Lambda^\star$, and equality (\ref{eq:primal_eq_dual}) is true.
\end{proof}
This theorem shows that when $\Sigma\in\mathcal{S}_{\hat{\Omega}}$, the relaxation is tight, i.e., the non-convex Frisch-Kalman problem can be exactly solved by the proposed convex approach in Algorithm~1.
\subsection{Refined Analysis}
The definition of $\mathcal{D}_{\hat{\Omega}}$ relies on the existence of a seemingly irrelevant matrix $\Lambda\in\mathbb{S}_0^n$. A desired characterization of $\Delta$ for the tight relaxation may be given by a neighborhood with a regular shape, such as a ball or a box in $\mathbb{D}^n$. To achieve this, we refine the set $\mathcal{D}_{\hat{\Omega}}$ in the following.

Let $\hat{\Omega}\in\mathbb{S}_+^n$ with $\text{\rm rank}(\hat{\Omega})=r$. Let $V\in\mathbb{R}^{n\times (n-r)}$ be an isometry onto the kernel space of $\hat{\Omega}$.

Define a linear operator associated with $\hat{\Omega}$ via
$$\bm{E}:~\mathbb{S}^{n-r}\to \mathbb{D}^n= X\mapsto \text{diag}(VXV^T).$$
\textcolor{black}{Define the following value associated with $\bm{E}$ via
$$\phi(\hat{\Omega}):=\inf_{{\Delta}\in\mathbb{D}_+^n\setminus\{0\}}\frac{\|\Delta\|_F}{\|\bm{E}^{\dagger}(\Delta)\|},$$
if $\bm{E}$ is surjective; otherwise, $\phi(\hat{\Omega})=0$. Here,
$$\bm{E}^{\dagger}=\bm{E}^*(\bm{E}\bm{E}^*)^{-1}:~\mathbb{D}^n \to \mathbb{S}^{n-r}$$ is the Moore-Penrose pseudoinverse for a surjective linear operator $\bm{E}$, where $\bm{E}^*$ is the adjoint of $\bm{E}$.}

Since $X\in\mathbb{S}^{n-r}$ is symmetric, the linear operator $\bm{E}$ is essentially from $\mathbb{R}^{{(n-r)(n-r+1)}/{2}}$ to $\mathbb{R}^n$. As a result, a necessary condition for its surjectivity is that $$\frac{(n-r)(n-r+1)}{2}\geq n,~\text{or}~r\leq \frac{2n+1-\sqrt{8n+1}}{2}.$$
In many cases, $n$ is known to be much larger than $r$ because the covariance matrix of the true data has a low rank, and hence the above inequality holds in general. On the other hand, when the inequality is satisfied, the linear operator $\bm{E}$ is surjective for almost all low-rank matrices $\hat{\Omega}$, i.e., $\phi(\hat{\Omega})>0$ in general.

Clearly, $\bm{E}$ can be non-surjective in terms of some extreme settings on $\hat{\Omega}$. For example, consider a rank-$r$ matrix
$$\hat{\Omega}=\begin{bmatrix}
                 S & 0 \\
                 0 & 0 \\
               \end{bmatrix}\in\mathbb{S}_+^n,
$$
where $S\in\mathbb{S}_{++}^r$ and $r<n$. It is easy to verify that $\bm{E}$ is not surjective on $\mathbb{D}^n$ since the identity matrix is not in its range, hence $\phi(\hat{\Omega})=0$.

With the assist of $\phi(\hat{\Omega})$, we have the following lemma, which simplifies the representation of set ${\mathcal{D}}_{\hat{\Omega}}$ and may refine the condition on the tight relaxation.
\begin{lemma}\label{lem:refined_analysis}
	Let $\hat{\Omega}\in\mathbb{S}_+^n$ with $\text{\rm rank}(\hat{\Omega})=r$. Then it holds
	$$\tilde{\mathcal{D}}_{\hat{\Omega}}:=\left\{{\Delta}\in\mathbb{D}_+^n\Big|~\|\Delta\|_F<\phi(\hat{\Omega})\sigma_r(\hat{\Omega})\right\}\subset{\mathcal{D}}_{\hat{\Omega}}.$$	
\end{lemma}
\begin{proof}
	Suppose $\bm{E}$ is surjective; otherwise, the statement is trivially true.
	As a result, it follows that for all matrices $\tilde{\Delta}\in\tilde{\mathcal{D}}_{\hat{\Omega}}$, there exists a $\Lambda\in\mathbb{S}^n_0$ such that
	$$V\bm{E}^{\dagger}(\tilde{\Delta})V^T=\Lambda+\tilde{\Delta}.$$
	It is clear that $\mathcal{R}(\Lambda+\tilde{\Delta})\perp\mathcal{R}(\hat{\Omega})$ since $$\mathcal{R}(\Lambda+\tilde{\Delta})\subset\mathcal{R}(V)=\mathcal{K}(\hat{\Omega})$$ from the above equality.
	On the other hand, note $\|\tilde{\Delta}\|_F<\phi(\hat{\Omega})\sigma_r(\hat{\Omega})$, then we have
	\begin{align*}
	\|\Lambda &+\tilde{\Delta}\|=\|\bm{E}^{\dagger}(\tilde{\Delta})\|\leq \sup_{{\Delta}\in\mathbb{D}_+^n\setminus\{0\}}\frac{\|\bm{E}^{\dagger}(\Delta)\|}{\|\Delta\|_F} \|\tilde{\Delta}\|_F\\
	&<\sup_{{\Delta}\in\mathbb{D}_+^n\setminus\{0\}}\frac{\|\bm{E}^{\dagger}(\Delta)\|}{\|\Delta\|_F} \phi(\hat{\Omega})\sigma_r(\hat{\Omega})=\sigma_r(\hat{\Omega}).
	\end{align*}
	Therefore, we can conclude that $\|\Lambda+\tilde{\Delta}\|<\sigma_r(\hat{\Omega})$ as well as $\mathcal{R}(\Lambda+\tilde{\Delta})\perp\mathcal{R}(\hat{\Omega})$, which shows $\tilde{\Delta}\in{\mathcal{D}}_{\hat{\Omega}}$.
\end{proof}
The set $\tilde{\mathcal{D}}_{\hat{\Omega}}$ is essentially the intersection of the nonnegative orthant and the open ball centred at $0$ with radius $\phi(\hat{\Omega})\sigma_r(\hat{\Omega})$ measured by the Euclidean distance in $\mathbb{R}^{n}$.
By Lemma~\ref{lem:refined_analysis}, we obtain the following corollary to Theorem~\ref{thm:tight_relax} immediately.
\begin{corollary}
Let $\phi(\hat{\Omega})>0$.	Then for all $\Sigma=\hat{\Omega}+\Delta>0$ with $\Delta\in\tilde{\mathcal{D}}_{\hat{\Omega}}$, \textcolor{black}{the solutions to (\ref{eq:relaxed_FK}) solve (\ref{eq:reformulated_FK})} with optimum attained on $\Omega^\star=\hat{\Omega}$.
\end{corollary}

\subsection{Case Study}
The following example is a case study for the proposed algorithm and the analysis on the tight relaxation. Consider the following rank-$1$ matrix
\begin{align}\label{eq:sigh_casestudy}
\hat{\Omega}=\begin{bmatrix}
16& 8 & 4\\
8 & 4 & 2\\
4  & 2 & 1\\
\end{bmatrix}=\begin{bmatrix}
4 \\ 2 \\ 1
\end{bmatrix}\begin{bmatrix}
4 & 2 & 1
\end{bmatrix}.
\end{align}
We start with the characterization of the set $\mathcal{D}_{\hat{\Omega}}$.
Through simple computation, we can represent the kernel space of $\hat{\Omega}=\mathcal{R}\left(V\right)$ where
$$V:=\begin{bmatrix}
0.49 & 0\\
-0.78 & -0.45\\
-0.39 & 0.89\\
\end{bmatrix}$$ is an isometry.
It is clear that for every $$\Delta=\text{diag}^*\left(\begin{bmatrix}d_1&d_2&d_3\end{bmatrix}\right)\in \mathcal{D}_{\hat{\Omega}},$$ there exists a matrix $\Lambda\in\mathbb{S}_0^3$ such that
$$\mathcal{R}(\Delta+\Lambda)\subset\mathcal{K}(\hat{\Omega})~\text{and}~\|\Delta+\Lambda\|<\sigma_1(\hat{\Omega}).$$
It is not hard to obtain the following parametrization
\begin{multline*}
\Delta+\Lambda=VXV^T\\
=\begin{bmatrix}
0.49 & 0\\
-0.78 & -0.45\\
-0.39 & 0.89\\
\end{bmatrix}\begin{bmatrix}
a & \frac{b}{\sqrt{2}}\\
\frac{b}{\sqrt{2}} & c\\
\end{bmatrix}\begin{bmatrix}
0.49 & 0\\
-0.78 & -0.45\\
-0.39 & 0.89\\
\end{bmatrix}^T.
\end{multline*}
Equating the diagonal terms on both the sides, we have
$$\begin{bmatrix}
d_1 \\ d_2 \\ d_3
\end{bmatrix}=\begin{bmatrix}
0.24 & 0 & 0\\
0.61  & 0.49 & 0.20\\
0.15 & -0.49 & 0.80\\
\end{bmatrix}\begin{bmatrix}
a \\ b \\ c\\
\end{bmatrix}=:E\begin{bmatrix}
a \\ b \\ c\\
\end{bmatrix}. $$
Here, matrix $E$ is invertible, which means for all diagonal matrices $\Delta$, there exists a $\Lambda\in\mathbb{S}_0^3$ such that $$\mathcal{R}(\Delta+\Lambda)\subset\mathcal{K}(\hat{\Omega}),$$
which is equivalent to
$$\mathcal{R}(\Delta+\Lambda)\perp\mathcal{R}(\hat{\Omega}).$$
\begin{figure}
	\centering
	\includegraphics[width=.52\textwidth]{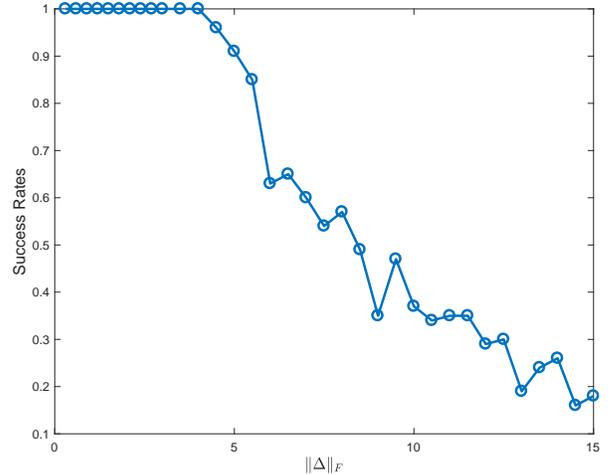}\\
	\caption{Success rates for Algorithm~1 to solve the Frisch-Kalman problem in (\ref{eq:FrischKalman}), where $\Sigma=\hat{\Omega}+\Delta>0$, $\hat{\Omega}$ is given in (\ref{eq:sigh_casestudy}), and $\Delta$ is generated according to (\ref{eq:generation_of_D}) for $T=100$ times with each level of $\|\Delta\|_F$.}
	\label{fig:succ_rate}
\end{figure}

Furthermore, we have the following inequalities
\begin{multline}\label{eq:case_study}
\|\Delta+\Lambda\|=\|VXV^T\|= \left\|X\right\| \leq \left\|X\right\|_F\\
=\left\|\begin{bmatrix}
a\\b\\c
\end{bmatrix}\right\|\leq \|E^{-1}\|\left\|\begin{bmatrix}
d_1 \\ d_2 \\ d_3
\end{bmatrix}\right\|=\|E^{-1}\|\|\Delta\|_F.
\end{multline}
As a result, as long as $\|\Delta\|_F<\|E^{-1}\|^{-1}{\sigma}_1(\hat{\Omega})=3.12$, it holds that
$\|\Delta+\Lambda\|<{\sigma}_1(\hat{\Omega})$. Due to the norm relaxation from $\|\cdot\|$ to $\|\cdot\|_F$ in (\ref{eq:case_study}), we know
 $$\phi(\hat{\Omega})\geq \|E^{-1}\|^{-1},$$
 and further
$$\{\Delta\in\mathbb{D}_+^3|~\|\Delta\|_F<3.12\} \subset\tilde{\mathcal{D}}_{\hat{\Omega}}\subset \mathcal{D}_{\hat{\Omega}}.$$
When the perturbation $\|\Delta\|_F$ is bounded by $3.12$, Algorithm~1 will solve the Frisch-Kalman problem for all $\Sigma=\hat{\Omega}+\Delta>0$ according to Theorem~\ref{thm:tight_relax}.

We perform repeated simulations based on different values of $\|\Delta\|_F$.  Specifically, fixing $\|\Delta\|_F>0$, we generate a diagonal matrices $\tilde{\Delta}$ satisfying
\begin{align}\label{eq:generation_of_D}
\tilde{\Delta}=\frac{\|\Delta\|_F}{\|{d}\|}\text{diag}^*({d}).
\end{align}
Here, vector $d\in\mathbb{R}^3$, and its elements $d_i$, $i=1,2,3$, are independent and identically distributed (i.i.d.), and uniformly distributed on $[0,1]$. Obviously, it holds that $\|\tilde{\Delta}\|_F=\|\Delta\|_F$, and we repeat the above procedures $T$ times for calculating success rates.

The success rates for Algorithm~1 to solve the Frisch-Kalman problem are shown in Fig.~\ref{fig:succ_rate}. When $\|\Delta\|_F<3.12$, we observe that the relaxations above (\ref{eq:relaxed_FK}) are indeed tight. When $\|\Delta\|_F$ increases, the chance when (\ref{eq:relaxed_FK}) solves the Frisch-Kalman problem decreases.

\section{Comparison with Existing Methods}\label{sec:compare_Heuristic}
Various heuristic methods have been investigated for solving rank minimization problems. In this section, we compare our proposed method with several mostly adopted existing methods on solving the Frisch-Kalman problem.
\subsection{Nuclear Norm Minimization}
In the context of factor analysis, nuclear norm (trace) minimization has been pursued as a suitable heuristic; see, for instance, \cite{shapiro1982rank,ning2015linear}. The nuclear norm of a matrix is defined as the sum of all its singular values. With this heuristic, the Frisch-Kalman problem is relaxed into
\begin{equation}\label{eq:nuclear_norm_heuristic}
\begin{aligned}
\min_{\Delta} \{\text{tr}(\Sigma-\Delta)|~{\Delta}\in\mathbb{D}_+^n,\Sigma\geq \Delta\geq 0\}.
\end{aligned}
\end{equation}
\textcolor{black}{One way to analyze the corresponding conditions on tight relaxation, i.e., when the solutions to (\ref{eq:nuclear_norm_heuristic}) solve the Frisch-Kalman problem, is via investigating the restricted isometry property (RIP) \cite{Recht2010} of an associated linear operator.} Define a linear operator $\bm{L}:~\mathbb{R}^{n\times n}\to\mathbb{R}^{n\times n}$ that projects a matrix $X$ onto its off-diagonal terms, i.e.,
$$\bm{L}=X\mapsto X-\text{diag}(X).$$
With this linear operator, we equivalently reformulate the Frisch-Kalman problem as follows. Given $\Sigma\in\mathbb{S}^n_{++}$, determine
\begin{align}\label{eq:RIP}
\min_{\Omega}\{\text{rank}(\Omega)|~\bm{L}(\Omega)=\bm{L}(\Sigma), \Sigma\geq \Omega\geq 0\}.
\end{align}
For every integer $r\in[1,n]$, define the $r$-restricted isometry
constant to be the smallest number $\alpha_r(\bm{L})$ such that
\begin{align}\label{eq:RIP_def}
\hspace{-8pt}(1-\alpha_r(\bm{L}))\|X\|_F \leq \|\bm{L}(X)\|_F \leq (1+\alpha_r(\bm{L}))\|X\|_F
\end{align}
holds for all matrices $X$ of rank at most $r$. Existing RIP conditions for the tight relaxation require that $\alpha_r(\cdot)<1$ for some $r\in[1,n]$; see, for instance, \cite[Theorems~3.2, 3.3]{Recht2010}. However, we can easily verify that $$1\geq \alpha_r(\bm{L})\geq\alpha_1(\bm{L})=1,~1\leq r\leq n,$$
by setting $X=I$ in (\ref{eq:RIP_def}).
Therefore,  the RIP conditions are not applicable to the Frisch-Kalman problem. Similar statements about the applicability of RIP conditions can be found in \cite{ning2015linear}.
\subsection{Low-Rank Inducing $r*$-norm}
A series of matrix norms, \textcolor{black}{called the $r*$-norms (or spectral $r$-support norms) \cite{mcdonald2016new,grussler2018low,grussler_pontus2018low}}, are defined by
\begin{align}\label{eq:lr_inducing_rstar}
\|M\|_{{l}_\infty,r*}:=\max_{\|X\|_{{l}_1,r}\leq 1} \langle X,M\rangle,
\end{align}
where $X,M\in\mathbb{R}^{m\times n}$, $r=1,2,\dots,\min\{m,n\}$, and $$\displaystyle \|X\|_{{l}_1,r}:=\sum_{k=1}^r\sigma_k(X)$$ is the Ky Fan $r$-norm. When $r=1$, $\|X\|_{{l}_1,1}$ reduces to the spectral norm and its dual norm $\|M\|_{{l}_\infty,1*}$ reduces to the nuclear norm. Therefore, the $r*$-norms include the well-known nuclear norm as a special case.
With these low-rank inducing norms, the Frisch-Kalman problem may be relaxed into
\begin{equation}\label{eq:rstar_heuristic}
\begin{aligned}
\min_{\Delta} \{\|\Sigma-\Delta\|_{{l}_\infty,r*}|~\Delta\in\mathbb{D}^n,\Sigma\geq \Delta\geq 0\}.
\end{aligned}
\end{equation}
Via similar developments in \cite{grussler_pontus2018low}, we can transform (\ref{eq:rstar_heuristic}) into the following SDP:
\begin{equation}\label{eq:rstar_heuristic_SDP}
\begin{aligned}
&\min_{W,\Delta,\gamma}\gamma\\
&\text{s.t.}~~W\in\mathbb{S}_+^n,~{\Delta}\in\mathbb{D}_+^n,~\Sigma-\Delta\in\mathbb{S}_+^n,\\
~&\begin{bmatrix}
    \gamma I-W & \Sigma-\Delta \\
    \Sigma-\Delta & I \\
  \end{bmatrix}\in\mathbb{S}_+^{2n},\\
~&\text{tr}(W)=\gamma(n-r).
\end{aligned}
\end{equation}
When applying it to the Frisch-Kalman problem, we search for the lowest-rank solution by sequentially solving (\ref{eq:rstar_heuristic_SDP}) with $r=1,2,\dots,n$.
\subsection{Log-Det Heuristic}
The logarithm of the determinant has been used as a smooth approximation for the rank function; see, for instance, \cite{fazel2003log}. For $X\in\mathbb{S}_{+}^n$, the function $\log\det(X+\delta I)$, where $\delta>0$, is used as a smooth surrogate for $\text{rank}(X)$. Since $\log\det(X+\delta I)$ is actually non-convex in $X$, local minimization methods are proposed in \cite{fazel2003log} by solving trace minimization problems iteratively. In this case, the Frisch-Kalman problem is approximately solved via the following iterations:
\begin{equation*}
\begin{aligned}
&\Delta_0=0,~\delta>0,W_k=(\Sigma-\Delta_k+\delta I)^{-1},\\
&\Delta_{k+1}=\argmin_{\Delta} \{\text{tr}(W_k(\Sigma-\Delta))|~\Delta\in\mathbb{D}^n,\Sigma\geq \Delta\geq 0\}.
\end{aligned}
\end{equation*}

\subsection{Simulation Result}
\begin{figure}
	\centering
	\includegraphics[width=.52\textwidth]{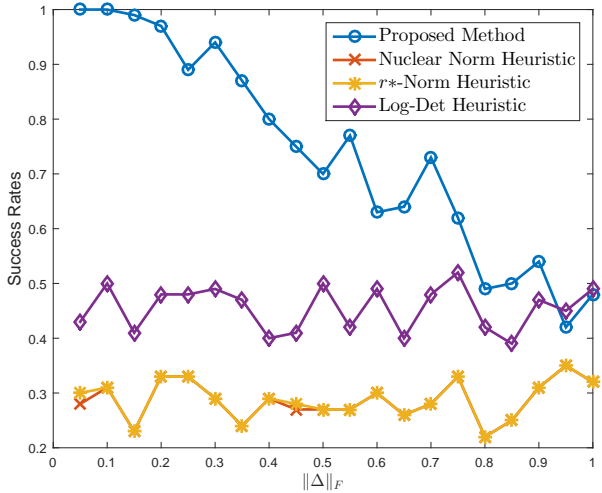}\\
	\caption{Success rates for the proposed, the nuclear norm, $r*$-norms and log-det heuristics to solve the Frisch-Kalman problem in (\ref{eq:FrischKalman}), respectively, where $\Sigma$ is randomly generated as described in the context with parameters $n=10$, $r=5$, and $T=100$. }
	\label{fig:succ_rate_nnh_r3}
\end{figure}

We compare the proposed algorithm with the existing methods, including the nuclear norm, $r*$-norms and log-det heuristics, based on randomly generated data for both $\hat{\Omega}$ and $\Delta$. The detailed randomization is given by the following steps.
\textcolor{black}{\begin{enumerate}
	\item Generate matrix $X\in\mathbb{R}^{r\times n}$ with $[X]_{ij}$ being standard i.i.d. Gaussian random variables, i.e., $[X]_{ij}\sim\mathbb{N}(0,1)$. Compute $\hat{\Omega}=X^T X$.
	\item Generate $\tilde{\Delta}$ with prescribed norm $\|\tilde{\Delta}\|_F=\|\Delta\|_F$ according to (\ref{eq:generation_of_D}) such that $\Sigma=\hat{\Omega}+\tilde{\Delta}>0$.
\end{enumerate}}
Given randomly generated matrices $\Sigma\in\mathbb{S}^n_{++}$, we check whether the heuristic methods will solve the Frisch-Kalman problem. For each level of $\|\Delta\|_F$, we repeat the the above procedures for $T$ times, and count for the success rates. \textcolor{black}{Here, the ``success rate'' refers to the percentage of experiments in which the recovered rank $r^\star$ satisfies that $r^\star \leq r$.}

It can be seen from Fig.~\ref{fig:succ_rate_nnh_r3} that the success rates
for the proposed method highly depend on the ``size of noise'', namely, the value $\|\Delta\|_F$, while those of the other heuristics do not. When $\|\Delta\|_F$ is close to zero, the success rate of Algorithm~1 approaches one. Practically, we may consider a suitable combination of all these heuristics.

\section{Conclusion and Future Work}\label{sec:conclusion}
A heuristic convex method is proposed for the century-old Frisch-Kalman problem. Both analytical and simulation results show that the method is accurate under the condition that the noise components are relatively small compared with the underlying data.

For future research, the proposed method may be improved via, for example, the combination with other heuristics, preprocessing on the observed data to remove possible outliers and so on. Another direction is to apply the method or the underlying ideas to solve more general rank minimization problems, such as the low-rank matrix completion, the data compression and so on.

\bibliographystyle{IEEEtran}

\bibliography{cdc2019}
\end{document}